\documentclass[letterpaper, 11 pt]{article}
\usepackage{fullpage,amsthm, amsmath, amssymb, amsfonts}
\usepackage{graphicx}

\newcommand{\ZZ}{\mathbb{Z}}

\newcommand{\RR}{\mathbb{R}}

\newcommand{\TP}{\mathbb{TP}}

\newcommand{\newword}[1]{\textbf{\emph{#1}}}

\newcommand{\M}{\mathcal{M}}

\newcommand{\T}{\mathcal{S}}
\newcommand{\Q}{\mathcal{Q}}

\def\M{\mathcal{M}}

\newtheorem{conj}{Conjecture}[section]

\newtheorem{theorem}[conj]{Theorem}

\newtheorem{proposition}[conj]{Proposition}

\newtheorem{lemma}[conj]{Lemma}

\newtheorem{corollary}[conj]{Corollary}
\newtheorem{question}[conj]{Question}

\theoremstyle{definition}
\newtheorem{definition}[conj]{Definition}

\newtheorem{remark}[conj]{Remark}

\begin{document}
 
\title{Triangulations of $\Delta_{n-1} \times \Delta_{d-1}$ and Tropical Oriented Matroids}
\author{Suho Oh \\
 Applied Mathematics Department \\
 Massachusetts Institute of Technology \\
 \texttt{suho@math.mit.edu} \\
 \\
 Hwanchul Yoo \\
 Mathematics Department \\
 Massachusetts Institute of Technology \\
 \texttt{hcyoo@math.mit.edu} \\ }



\date{}
\maketitle

\abstract{Develin and Sturmfels showed that regular triangulations of $\Delta_{n-1} \times \Delta_{d-1}$ can be thought as tropical polytopes. Tropical oriented matroids were defined by Ardila and Develin, and were conjectured to be in bijection with all subdivisions of $\Delta_{n-1} \times \Delta_{d-1}$. In this paper, we show that any triangulation of $\Delta_{n-1} \times \Delta_{d-1}$ encodes a tropical oriented matroid. We also suggest a new class of combinatorial objects that may describe all subdivisions of a bigger class of polytopes.}


\section{Introduction}

Studying triangulations of product of simplices is a very active field of research and there have been numerous results being tied to many different fields (\cite{Ardila2007495}, \cite{MR1637892},\cite{MR1231196},\cite{0827.14036},\cite{MR2131129},\cite{MR1098809},\cite{MR2013970},\cite{0878.52004},\cite{0958.52017},\cite{MR2181765}).

In \cite{MR2054977}, Develin and Sturmfels showed that regular triangulations can be thought as tropical polytopes. Tropical polytopes are essentially tropical hyperplane arrangements. Ardila and Develin defined tropical oriented matroids, that generalize tropical hyperplane arrangements \cite{2007arXiv0706.2920A}. And they conjectured that tropical oriented matroids are essentially the same as subdivisions of product of simplices. In oriented matroid theory, it is a very well known result that realizable oriented matroids come from hyperplane arrangements and oriented matroids in general come from pseudo-sphere arrangements. They showed that a tropical oriented matroid encodes a subdivision. They also showed that a triangulation of $\Delta_{n-1} \times \Delta_2$ enocodes a tropical oriented matroid. In this paper, we provide a strong evidence for the conjecture, by showing that a triangulation of $\Delta_{n-1} \times \Delta_{d-1}$ encodes a tropical oriented matroid.



In section $2$, we go over the basics of triangulations of $\Delta_{n-1} \times \Delta_{d-1}$, fine mixed subdivisions of $n\Delta_{d-1}$ and develop some tools. In section $3$, we go over the definition of tropical oriented matroids. In section $4$, we show that the collection of trees in a fine mixed subdivision of $n\Delta_{d-1}$ satisfies the elimination property. In section $5$, we suggest a new class of objects that may describe all subdivisions of a generalized permutohedra.



\medskip

\textbf{Acknowledgment} We would like to thank Alexander Postnikov, Federico Ardila and C\'esar Ceballos for useful discussions. We would also like to thank Michel Goemans for suggesting Theorem~\ref{thm:int}.

\section{Triangulations of $\Delta_{n-1} \times \Delta_{d-1}$ and Fine Mixed Subdivisions of $n \Delta_{d-1}$}

Each full-dimensional simplex in a triangulation of $\Delta_{n-1} \times \Delta_{d-1}$ can be described by a spanning tree of the bipartite graph $K_{n,d}$. To see this, we label the vertices of $\Delta_{n-1}$ with $[n]$ and vertices of $\Delta_{d-1}$ with $[d]$, then each vertex of $\Delta_{n-1} \times \Delta_{d-1}$ corresponds to an edge of $K_{n,d}$. We will say that in $K_{n,d}$, the vertices corresponding to $\Delta_{n-1}$ are on the left side and the vertices corresponding to $\Delta_{d-1}$ are on the right side. The vertices of each subpolytope in $\Delta_{n-1} \times \Delta_{d-1}$ determine a subgraph of $K_{n,d}$. We use $(A_1,\cdots,A_n)$ where $A_1,\cdots,A_n \subseteq [d]$, to denote a subgraph of $K_{n,d}$ that has edges $(i,j)$ for each $j \in A_i$.

Via the Cayley trick, one can think of a triangulation of $\Delta_{n-1} \times \Delta_{d-1}$ as a fine mixed subdivision of $n \Delta_{d-1}$ \cite{SantosCayley}. We will first go over the basics of fine mixed subdivisions, then state some properties that will be useful for our purpose.
  
\begin{definition}[\cite{Postnikov01012009}]
Let $r$ be the dimension of the Minkowski sum $P_1 + \cdots + P_n$. A \newword{Minkowski cell} in this sum is a polytope $B_1 + \cdots + B_n$ of dimension $r$ where $B_i$ is the convex hull of some subset of vertices of $P_i$. A \newword{mixed subdivision} of the sum is the decomposition into union of Minkowski cells such that intersection of any two cells is their common face. A mixed subdivision is \newword{fine} if there is no refinement possible.
\end{definition}

We define the simplex $\Delta_{d-1}$ as the convex hull of points $(1,0,\cdots,0),(0,1,\cdots,0),\cdots,(0,\cdots,1)$ in $\RR^d$. In this paper, we are only interested in fine mixed subdivisions of $\Delta_{d-1} + \cdots + \Delta_{d-1}$. 

\begin{lemma}[\cite{SantosCayley}]
\label{lem:product}
A mixed subdivision is fine if and only if, for each mixed cell $B = B_1 + \cdots + B_n$ in this subdivision, all $B_i$ are simplices and $\sum dim B_i = dim B$.
\end{lemma}
 
The lemma tells us that each fine cell $B_1 + \cdots + B_n$ is isomorphic to the direct product $B_1 \times \cdots \times B_n$ of simplices. Let $I_i$ be the set of vertices of $B_i$. We think of each cell as a subgraph $(I_1,\cdots,I_n)$ and this is a spanning tree \cite{Postnikov01012009}.

\begin{remark}
\label{rem:subdivsurrounding}
The above lemma also tells us that if we take $J_i \subseteq I_i, J_i \not = \emptyset$ for each $i$, then $(J_1,\cdots,J_n)$ encodes a face of this cell. From now on, we will use the subgraph of $K_{n,d}$ and its corresponding face interchangeably. That is, a face $(J_1,\cdots,J_n)$ means a face $\Delta_{J_1} + \cdots + \Delta_{J_n}$.
\end{remark}

To avoid confusion with the tropical oriented matroid terminology, we call the $0$-dimensional faces as \newword{topes}. For two trees $T$ and $T'$ of $K_{n,d}$, let $U(T,T')$ be the directed graph which is the union of edges of $T$ and $T'$ with edges of $T$ oriented from left to right and edges of $T'$ oriented from right to left. A directed \newword{cycle} is a sequence of directed edges $(i_1,i_2),(i_2,i_3),\cdots,(i_{k-1},i_k),(i_k,i_1)$ such that all $i_1,\cdots,i_k$ are distinct. Now we can say exactly which set of spanning trees describes a fine mixed subdivision of $n\Delta_{d-1}$.

\begin{theorem}[\cite{Santos97triangulationsof},\cite{Ardila2007495}]
A collection of subgraphs $T_1,\cdots,T_k$ of $K_{n,d}$ encodes a fine mixed subdivision of $n\Delta_{d-1}$ if and only if:
\begin{enumerate}
\item Each $T_i$ is a spanning tree of $K_{n,d}$.
\item For each $T_i$ and each edge $e$ of $T_i$, either $T_i \setminus e$ has an isolated vertex or there is another $T_j$ containing $T_i \setminus e$.
\item For any pair $i,j$ of $[n]$, there is no cycle in $U(T_i,T_j)$.
\end{enumerate}
\end{theorem}

Given any subgraph $T$ of $K_{n,d}$, define the \newword{left degree vector (LDV)} $ld(T)=(d_1-1,\cdots,d_n-1)$ where $d_i$ is the degree of the vertex $i \in [n]$ on the left side of $T$. Similarly, define the \newword{right degree vector (RDV)} $ld(T)=(d_1-1,\cdots,d_r-1)$ where $d_i$ is the degree of the vertex $i \in [d]$ on the right side of $T$. The following proposition is a special case of a statement in the proof of Theorem 11.3 in \cite{Postnikov01012009}.

 
\begin{proposition}[\cite{Postnikov01012009}]
\label{prop:Gcell}
Fix a fine mixed subdivision of $n\Delta_{d-1}$. Let $T_1,\cdots,T_s$ be the collection of cells. Then the map ${T_i} \rightarrow ld(T_i)$ is a bijection between fine cells in this subdivision and the set of sequences $(a_1,\cdots,a_n)$ satisfying $\sum a_i = d-1$ and $a_i \geq 0$ for all $i \in [n]$. The same holds for the map $T_i \rightarrow rd(T_i)$.
\end{proposition}
 
The reason we are interested in LDV and RDV is because LDV governs the shape of the cell and RDV governs the location of the cell.

Given $n\Delta_{d-1}$ and $i\in [d]$, we call the facet opposite to vertex $i$ as the \newword{$i$-facet}. A length $n$ simplex in a plane can be filled with upper and lower triangles. In higher dimension, although there is no analogue for the lower triangles, there is one for the upper triangles. It is just the collection of length $1$ simplices that have integer coordinates. We call these simplices the \newword{unit simplices}. We express the \newword{location} of a unit simplex as $(a_1,\cdots,a_d)$, where $a_i \in \ZZ$ stands for the distance between the $i$-facet and the unit simplex. We also have the relation that $\sum_i a_i = n-1$. See Figure~\ref{fig:cellrdv} for an example. The following lemma is a direct consequence of Lemma 14.9 of \cite{Postnikov01012009}.

\begin{lemma}
\label{lem:unit}
Each cell $T=(T_1,\cdots,T_n)$ in the fine mixed subdivision of $n\Delta_{d-1}$ contains exactly one unit simplex. The location of such simplex is equal to $rd(T)$.
\end{lemma}

An example of this phenomenon is given in Figure~\ref{fig:cellrdv}.

\begin{figure}[htbp]
	\centering
		\includegraphics[width=0.3\textwidth]{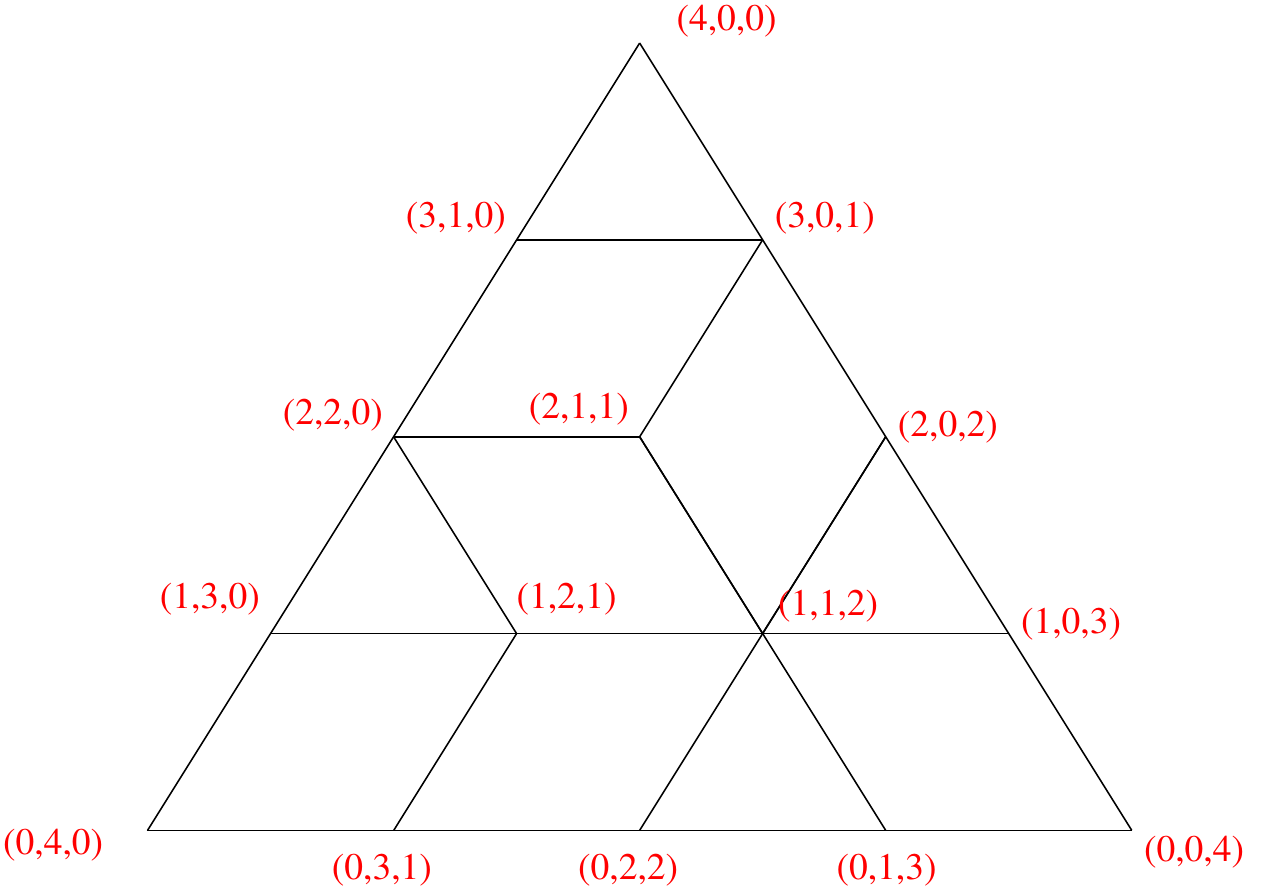}
		\includegraphics[width=0.3\textwidth]{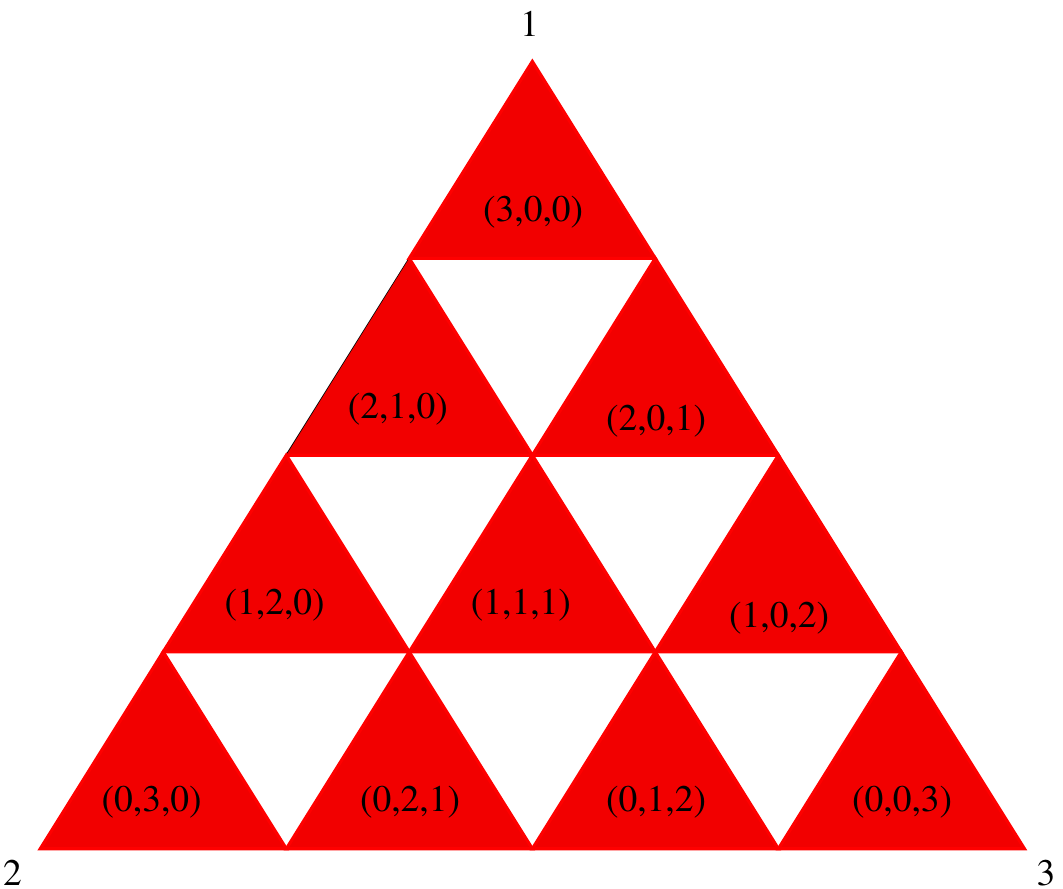}
		\includegraphics[width=0.3\textwidth]{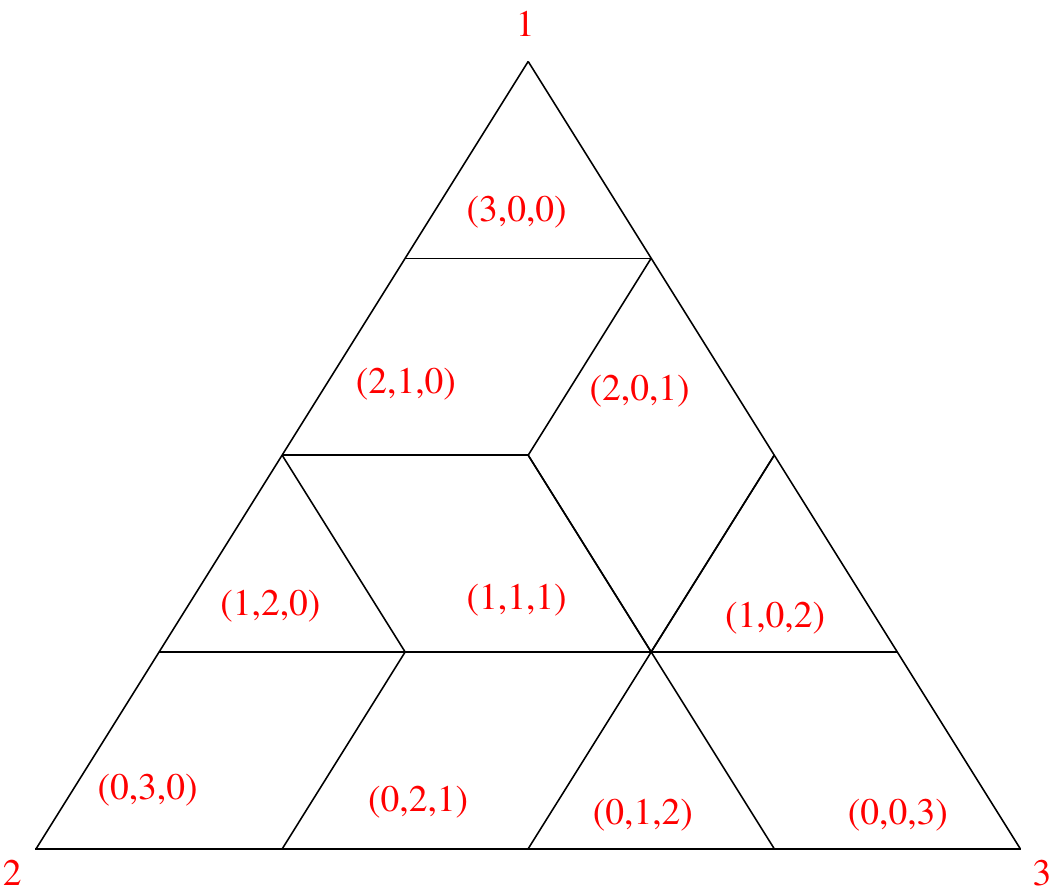}
	\caption{The number of $i$'s in a tope for each $i \in [d]$ describes the position of the tope. RDV describes the position of the unit simplex that the cell contains.}
	\label{fig:cellrdv}
\end{figure}

\section{Tropical Oriented Matroids}
In this section, we will review the definition of tropical hyperplane arrangements and tropical oriented matroids that were defined in \cite{2007arXiv0706.2920A}.

\begin{definition}
The \newword{tropical semiring} is given by the real numbers $\RR$ together with the operations of tropical addition $\oplus$ and tropical multiplication $\odot$ defined by $a \oplus b = max(a,b)$ and $a \odot b = a+b$.
\end{definition}

For convenience, we will work in the tropical projective $(d-1)$-space $\TP^{d-1}$, given by modding out by tropical scalar multiplication. In this space, \newword{tropical hyperplanes} are given by the vanishing locus of $\bigoplus c_i \odot x_i$, where the vanishing locus is defined to be the set of points where $max(c_1 + x_1,\cdots,c_d+x_d)$ is achieved at least twice.

Given an arrangement $H_1,\cdots,H_n$ in $\TP^{d-1}$, the \newword{type} of a point $x \in \TP^{d-1}$ is the $n$-tuple $(A_1,\cdots,A_n)$, where $A_i \subseteq [d]$ is the set of closed sectors of the hyperplane $H_i$ which $x$ is contained in. And since all points in a face of the arrangement have the same type, that type is called the type of the face.

\begin{definition}
An $(n,d)$-type is an $n$-tuple $A=(A_1,\cdots,A_n)$ of nonempty subsets of $[d]:=\{1,\cdots,d\}$. The sets $A_1,\cdots,A_n$ are called the \newword{coordinates} of $A$. 
\end{definition}

One should keep in mind that these types will correspond to trees coming from the faces of a triangulation of $\Delta_{n-1} \times \Delta_{d-1}$.

\begin{definition}[\cite{2007arXiv0706.2920A}]
Given two $(n,d)$-types $A$ and $B$, the \newword{comparability graph} $CG_{A,B}$ has vertex set $[d]$. For $1 \leq i \leq n$, we draw an edge between $j$ and $k$ for each $j \in A_i$ and $k \in B_i$. That edge is undirected if $j,k \in A_i \cap B_i$, and it is directed $j \rightarrow k$ otherwise.
\end{definition}

\begin{definition}[\cite{2007arXiv0706.2920A}]
A \newword{semidigraph} is a graph with some undirected edges and some directed edges. A \newword{directed path} from $a$ to $b$ in a semidigraph is a collection of vertices $v_0 = a, v_1,\cdots,v_k=b$ and a collection of edges $e_1,\cdots,e_k$, at least one of which is directed, such that $e_i$ is either a directed edge from $v_{i-1}$ to $v_i$ or an undirected edge connecting the two. A \newword{directed cycle} is a directed path with identical endpoints. A semidigraph is \newword{acyclic} if it has no directed cycles.
\end{definition}

\begin{definition}[\cite{2007arXiv0706.2920A}]
\label{def:bad}
The \newword{refinement} of a type $A=(A_1,\cdots,A_n)$ with respect to an ordered partition $P=(P_1,\cdots,P_r)$ of $[d]$ is $A_P = (A_1 \cap P_{m(1)},\cdots,A_n \cap P_{m(n)})$ where $m(i)$ is the largest index for which $A_I \cap P_{m(i)}$ is non-empty. A refinement $A_P$ is \newword{total} if all of its entries are singletons.
\end{definition}

For readers that are confused with this definition, one can ignore this definition and just think of the refinement as taking any nonempty subset of each $A_i$, since we will only be cosidering tropical oriented matroids corresponding to triangulations.

\begin{definition}[\cite{2007arXiv0706.2920A}]
A \newword{tropical oriented matroid} $M$ (with parameters $(n,d)$) is a collection of $(n,d)$-types which satisfy the following four axioms:
\begin{itemize}
\item Boundary : For each $j \in [d]$, the type $\newword{j} :=(j,\cdots,j)$ is in $M$.
\item Elimination : If we have two types $A$ and $B$ in $M$ and a position $j \in [n]$, then there exists a type $C$ in $M$ with $C_j = A_j \cup B_j$, and $C_k \in \{A_k,B_k,A_k \cup B_k\}$ for all $k \in [n]$.
\item Comparability : The comparability graph $CG_{A,B}$ of any two types $A$ and $B$ in $M$ is acyclic.
\item Surrounding : If $A$ is a type in $M$, then any refinement of $A$ is also in $M$.
\end{itemize}
\end{definition}

\begin{theorem}[\cite{2007arXiv0706.2920A}]
The types of the vertices of a tropical oriented matroid $M$ with parameters $(n,d)$ describe a set of spanning graphs defining a mixed subdivision of $n \Delta_{d-1}.$
\end{theorem}

They proposed the following three conjectures:
\begin{enumerate}
\item There is a one-to-one correspondence between the set of spanning graphs defining a subdivision of $\Delta_{n-1} \times \Delta_{d-1}$ and a tropical oriented matroid with parameters $(n,d)$.
\item The dual of a tropical oriented matroid with parameters $(n,d)$ is a tropical oriented matroid with parameters $(d,n)$.
\item Every tropical oriented matroid can be realized by an arrangement of tropical pseudo-hyperplanes.
\end{enumerate}


Before we end this section, we are going to present an easier way to think of the surrounding axiom.

\begin{lemma}
\label{lem:surrounding}
Let $A = (A_1,\cdots,A_n)$ be a $(n,d)$-type of a tropical oriented matroid, such that if one views this type as a subgraph of $K_{n,d}$, then it does not contain a cycle. Choose any $i \in [d]$ such that $|A_i| > 1$. Then choose any $k \in A_i$. Let $A'$ be obtained from $A$ by deleting $k$ from $A_i$. Then surrounding axiom tells us that $A'$ is in this tropical oriented matroid.
\end{lemma}
\begin{proof}
Let $Z$ be the union of all $A_j$ such that $k \in A_j$ and $j \not = i$. Let $W$ be the union of rest of $A_j$'s. Then $Z \cap W = \{k\}$ since otherwise, we get a cycle in $A$. So let our ordered partition be $(W^c \cup \{k\}, W \setminus \{k\})$. Then we get $A'$ from $A$ by a refinement as given in Definition~\ref{def:bad}.
\end{proof}

This is a more natural way to think of the surrounding axiom for our purpose, since all types coming from a fine mixed subdivision of $n \Delta_{d-1}$ have no cycles and satisfy this property, as can be seen from Remark~\ref{rem:subdivsurrounding}. Whenever we use this property (or Remark~\ref{rem:subdivsurrounding}), we will refer to this as the \newword{surrounding property}.


\section{Elimination Property}

Fix a fine mixed subdivision of $n\Delta_{d-1}$. Let $\M$ denote the collection of trees coming the subdivision. We are going to show that this is a tropical oriented matroid. Although we don't use it, our proof is heavily motivated from the topological representation conjecture that a mixed subdivision of $n\Delta_{d-1}$ can be viewed as a tropical pseudo-hyperplane arrangement.

Roughly, the elimination property can be thought as existence of a very nice path between two types $A$ and $B$. In particular, if $A_i = B_i$, we want a path such that its $i$-th coordinate is always equal to $A_i=B_i$. We are going to use induction based on an index defined for each pair of types, called \newword{rank}. Throughout the examples given in the section, for convenience, we are going to write sets such as $\{1,2,3\}$ by $123$. Also, recall that we call the $0$-dimensional faces in a fine mixed subdivision of $n\Delta_{d-1}$ as \newword{topes}, instead of vertices, to avoid confusion with the tropical oriented matroid terminology. 

Here is a motivation for the definition of the rank. Assume we are given a fine mixed subdivision of $n\Delta_2$ and let $A$ and $B$ be two types such that $A_i = B_i$. Let's look at the corresponding tropical pseudo-hyperplane arrangement. We are going to consider the case when $A_i = B_i = \{2\}$ and this is illustrated in Figure~\ref{fig:elimex}. Assume we are given some path between $A$ and $B$ such that for some types along this path, the $i$-th coordinate is not equal to $A_i$. Let $C$ and $D$ be the first and last points at which the path intersects the $i$-th tropical pseudo-hyperplane. Then $C$ and $D$ are both on the boundary of the region $\{2\}$ with respect to the $i$-th tropical pseudo-hyperplane. If we know that there is a nice path between these two points on this boundary, then we can lift this path a little bit to get a path inside the $\{2\}$-region by using the surrounding property.

\begin{figure}[htbp]
	\centering
		\includegraphics[width=0.2\textwidth]{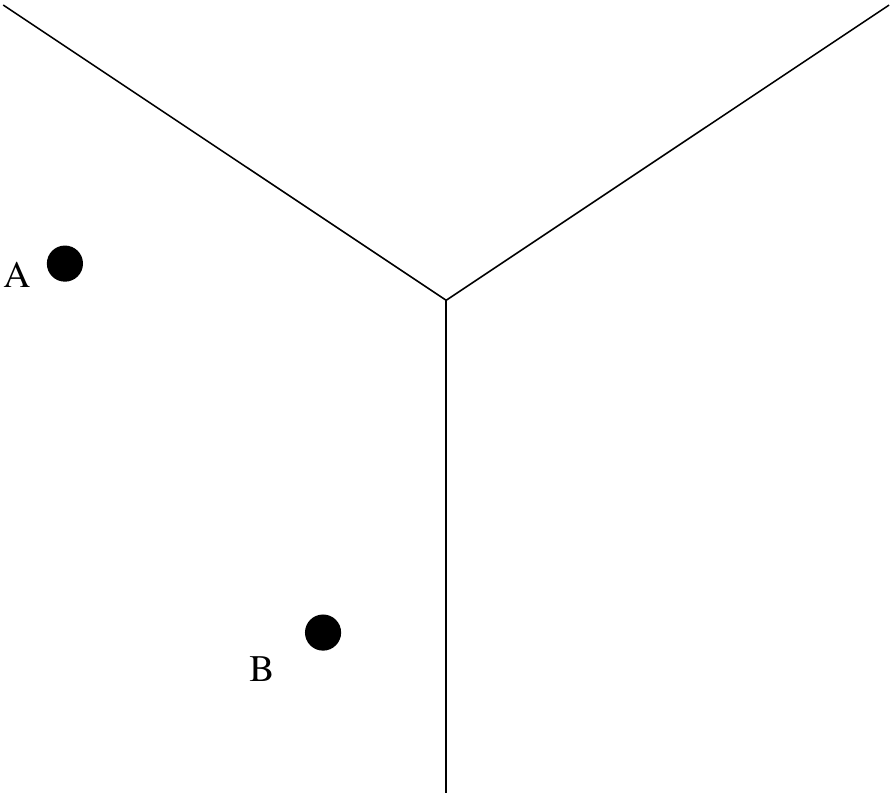}
		\includegraphics[width=0.2\textwidth]{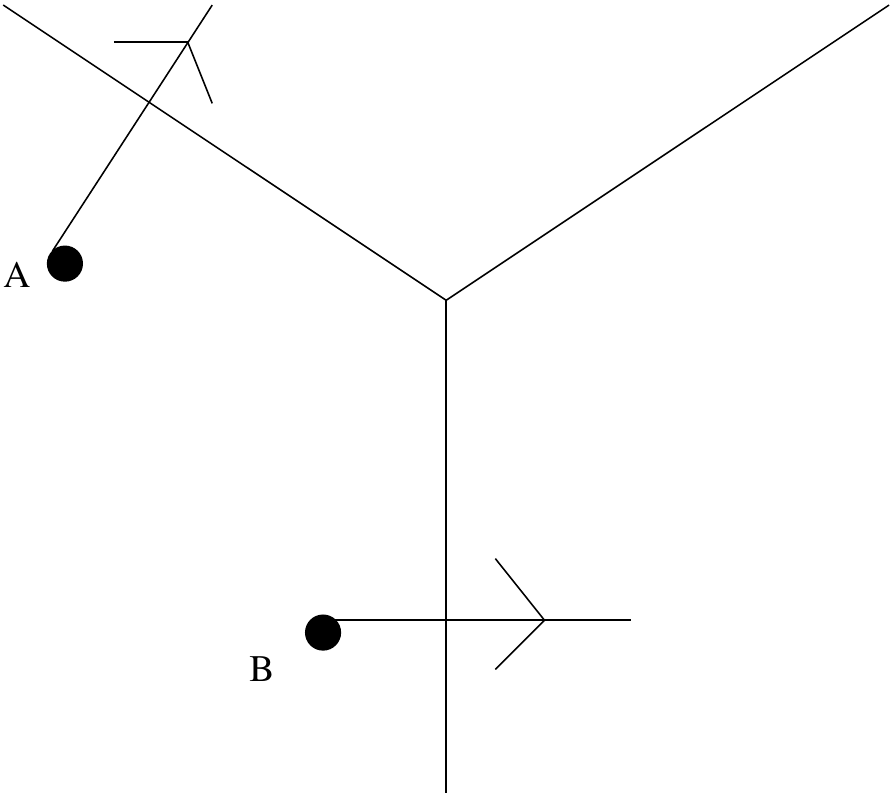}
		\includegraphics[width=0.2\textwidth]{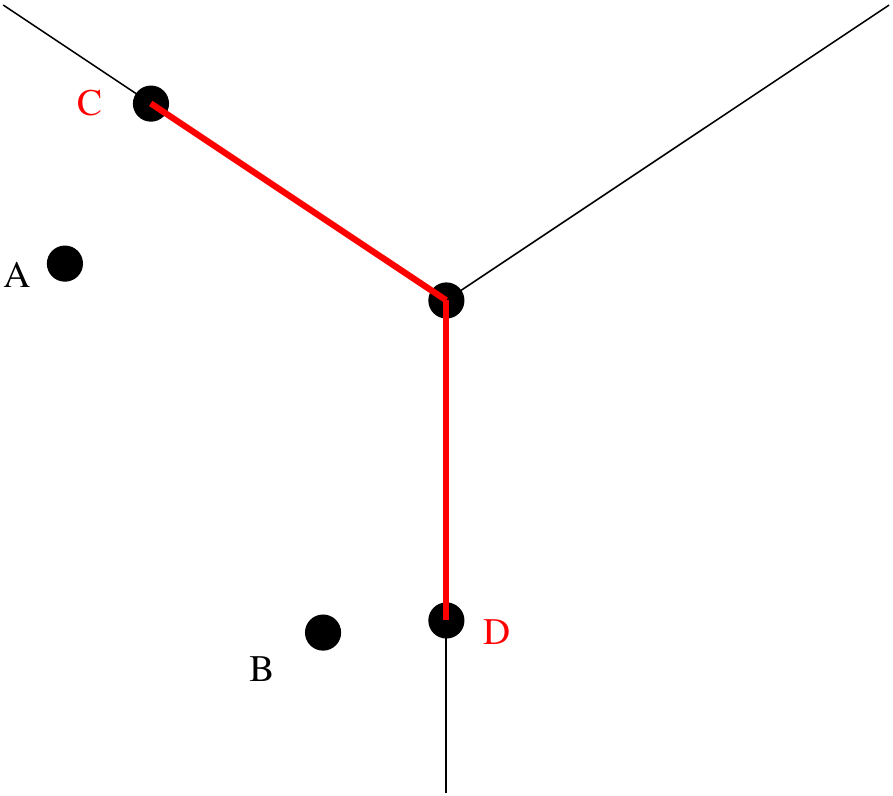}
		\includegraphics[width=0.2\textwidth]{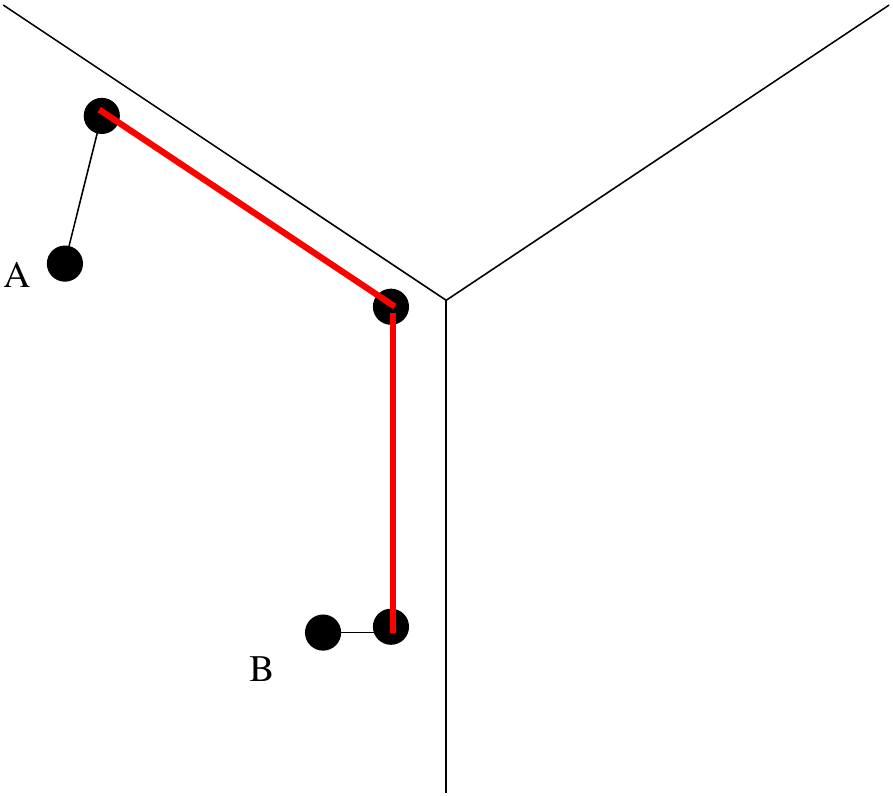}
	\caption{Rank is a good index for proving elimination property}
	\label{fig:elimex}
\end{figure}

\begin{definition}
For each pair of types $A,B \in \M$, the \newword{rank} of the pair $(A,B)$ is defined as $(r_1,\cdots,r_n)$ where for each $i \in [d]$, $r_i = min(|A_i|,|B_i|) - 1$. This is going to be denoted by $rk(A,B)$. For any $\alpha = (a_1,\cdots,a_n)$ and $\beta = (b_1,\cdots,b_n)$, we write $\alpha \geq \beta$ if we have $a_i \geq b_i$ for all $i \in [n]$. Similarly we write $\alpha > \beta$ if the inequality is strict in at least one coordinate.
\end{definition}

For example, $rk((123,2),(3,123)) = (0,0)$ and $rk((123,45),(3,125)) = (0,1)$. 

So $rk((123,2),(3,123)) < rk((123,45),(3,125))$.

\begin{definition}
We will say that two types $A$ and $B$ are \newword{adjacent} if $A$ and $B$ are different in exactly one coordinate and also differ by one element in it. A \newword{path} between two types is a sequence of types $A = C^0 \rightarrow C^1 \rightarrow C^2 \rightarrow \cdots \rightarrow C^{q-1} \rightarrow C^q = B$ such that each $C^t$ is adjacent to $C^{t-1}$ and $C^{t+1}$. The \newword{length} of the path is given by $q$. Given a path, we say that coordinate $i$ is \newword{strong} if:
\begin{enumerate}
\item in that coordinate, after some some element is deleted, no element gets added.
\item for all $t$, we have $A_i \cup B_i \supseteq C^t_i$.
\item if an element $j$ was added, then it does not get deleted later. This implies that $C_i^t \subseteq A_i \cup B_i$ for all $t$.
\end{enumerate}

A \newword{strong path} between types $A$ and $B$ is a path that is strong in every coordinates.

\end{definition}

A strong path is a path such that in each coordinate, it changes like
$$123 \rightarrow 1234 \rightarrow 12345 \rightarrow 1245 \rightarrow 145.$$

The reason we are interested in strong paths is because it is enough to find a strong path between any two types $A$ and $B$  to prove the elimination property for $\M$.

\begin{lemma}
\label{lem:pathconvert}
If there is a strong path between any two types $A$ and $B$, then elimination holds.
\end{lemma}

\begin{proof}
Given a strong path between $A$ and $B$, we have that for each coordinate $i$, 
\begin{itemize}
\item there is type $C^t$ on the path such that $C_i^t = A_i \cup B_i$,
\item for any type $C^t$ on the path, $C_i^t$ contains $A_i$ or $B_i$.
\end{itemize}

Hence the result follows from the surrounding property.
\end{proof}

Notice that in the example of a strong path above, the cardinality of each set is bounded below by $min(|A_i|,|B_i|)$. When we are looking for a strong path between $A$ and $B$, we do not consider all types. We only consider the types where the cardinality is bounded below by $rk(A,B)$.

\begin{definition}
For each $\alpha = (a_1,\cdots,a_n)$, $\Q_{\alpha}$ is defined as the collection of types $(A_1,\cdots,A_n)$ such that $|A_i| > a_i$ for all $i \in [n]$. 
\end{definition}

We use $\Delta(A,B)$ to denote $\sum_i (|A_i \setminus B_i | + |B_i \setminus A_i|)$. Then any path between $A$ and $B$ has length at least $\Delta(A,B)$. The length of a strong path between $A$ and $B$ is equal to $\Delta(A,B)$. We are later going to show that we can transform a lengthwise-shortest path between $A$ and $B$ in $\Q_{rk(A,B)}$ to a strong path. So we want to show that given any types $A$ and $B$ in $\Q_{\alpha}$, there is a path connecting them in $\Q_{\alpha}$.

\begin{figure}[htbp]
	\centering
		\includegraphics[width=0.4\textwidth]{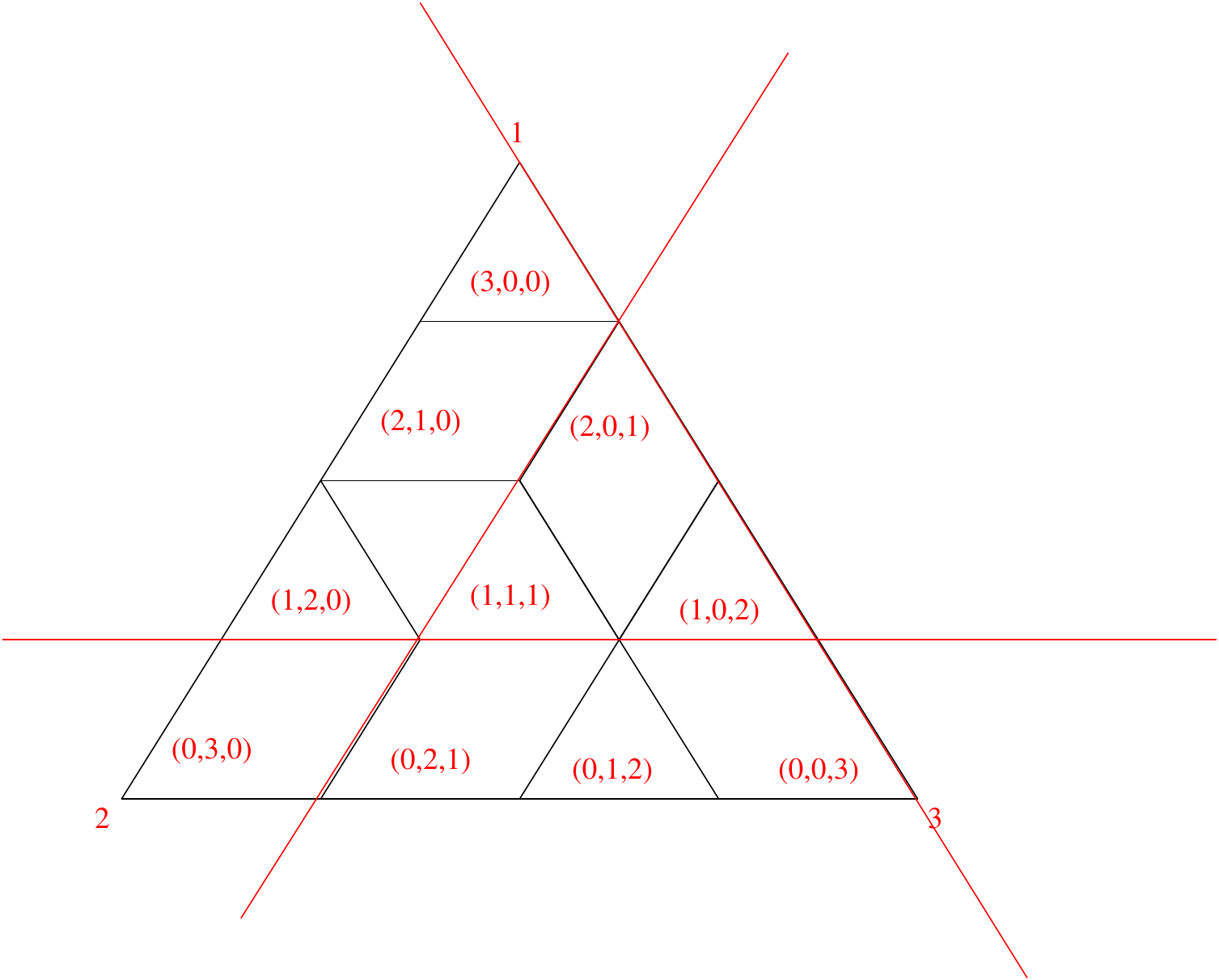}
		\includegraphics[width=0.4\textwidth]{msubtope}
	\caption{How $\T_{1,0,1}^*$ looks like.} 
	\label{fig:toperank}
\end{figure}

If we consider only the cells of $\Q_{\alpha}$, we are basically putting a cardinality restriction on the LDV's. Given an $(n,d)$-type $A$ such that all $i \in [d]$ appears in $A$, we define $A^T$ as a type with parameters $(d,n)$, that has $i \in A_j^T$ if and only if $j \in A_i$. We say that $A^T$ is the \newword{dual} of $A$. We can take the dual of any type that is not on the boundary of $n\Delta_{d-1}$. Dual of a cell is a cell, LDV becomes RDV, and the cardinality restriction on RDV is easy to view.

\begin{remark}
\label{rem:dual}
Define $\Q_{\alpha}^*$ to be the collection of types in the fine dual mixed subdivision (i.e. mixed subdivision of $d \Delta_{n-1}$ coming from the same triangulation $\Delta_{d-1} \times \Delta_{n-1}$) such that it contains strictly more than $a_i$ number of $i$'s for each $i \in [n]$. Then a cell is in $\Q_{\alpha}$ if and only if its dual is in $\Q_{\alpha}^*$. And two cells are adjacent (i.e. sharing a common facet) in $\Q_{\alpha}$ if and only if their duals are adjacent in $\Q_{\alpha}^*$. 
\end{remark}

Due to Lemma~\ref{lem:unit}, the unit simplices in these cells form a subsimplex of $d \Delta_{n-1}$. We will denote this as $\T_{\alpha}^*$. Although we will not define $\T_{\alpha}$, we will keep the star in the notation to emphasize the fact that $\T_{\alpha}^*$ is in $\Q_{\alpha}^*$. Take a look at the first picture of Figure~\ref{fig:toperank}. Cells in $\Q_{1,0,1}^*$ are the two rhombi and the simplex having RDV $(1,1,1),(2,0,1)$ and $(1,0,2)$. Also, $\T_{1,0,1}^*$ is the length $2$ simplex surrounded by the red lines. Before we prove that $\Q_{\alpha}$ is connected, we need the following well known results in integer programming.

\begin{theorem}[\cite{Schrijver03acourse}]
\label{thm:int}
A matrix $Y$ is called \newword{totally unimodular} if each square submatrix of $Y$ has determinant equal to $0, +1$ or $-1$. Let $Y$ be a totally unimodular $m \times n$ matrix and let $b \in \ZZ^m$. Then the polyhedron $P:=\{x|Yx\leq b\}$ has integer vertices.
\end{theorem}

\begin{lemma}[\cite{citeulike:323941}]
\label{lem:int}
A matrix $Y$ is an \newword{interval matrix} if it is a $\{0,1\}$-matrix and each row of $Y$ has $1$'s consecutively. Then $Y$ is also totally unimodular.
\end{lemma}

We want to show that the matrix defining a fine cell is totally unimodular. To do this, we need a way to describe the matrix defining a fine cell. Note that $n\Delta_{d-1}$ lives on the plane $x_1+\dots+x_d=n$. Let us project it onto the plane $x_d=0$. Denote the image as $n\Delta_{d-1}'$, which lives in $\RR^{d-1}$. The projection does not change any fine mixed subdivision structure.

\begin{lemma}
Let $T$ be a fine mixed cell of $n\Delta_{d-1}'$. For any edge $e$ of $T$ that is not connected to a leaf on the left side, we assign a facet $F_e$ of $T$ by deleting the edge from $T$. Let us denote by $I_e$ the set of vertices on the right side which are not connected to $d$ in $T\setminus e$. The equation of $F_e$ is given by $\sum_{j\in I_e} x_j = c$ for some $c \in \ZZ$.

\end{lemma}

\begin{proof}
Let us denote the type of $F_e$ by $(J_1,\dots,J_n)$. Without loss of generality, we can assume that $J_k$'s are ordered so that $\cup_{k=1}^m J_k = I_e$ and $\cup_{k=m+1}^n J_k = [d]\setminus I_e$. 


Denoting the coordinate vectors as $e_j$'s, any vector lying on $F_e$ is of the form 
$$\sum_{k=1}^m \sum_{j\in J_k} (\lambda_{k,j}-\lambda'_{k,j}) e_j + \sum_{k=m+1}^n \sum_{j\in J_k\setminus d} (\lambda_{k,j}-\lambda'_{k,j}) e_j$$
where $\sum_{j\in J_k} \lambda_{k,j}= \sum_{j\in J_k} \lambda'_{k,j}=1$ for all $k\le m$. By   construction, $e_j$'s appearing in the first summation are precisely those for $j\in I_e$, and $e_j$'s appearing in the second summation are those for $j\notin I_e$. Therefore $F_e$ is clearly orthogonal to $(n_1,\dots,n_{d-1})$, where $n_i = 1$ if $i \in I_e$ and $0$ otherwise.
\end{proof}


\begin{corollary}
\label{cor:int}
A matrix $Y$ defining a fine mixed cell $T$ in $n\Delta_{d-1}$ is totally unimodular.
\end{corollary}

\begin{proof}
From the way the projection was defined, it is enough to show that the matrix $Y'$ defining a cell in $n\Delta_{d-1}'$ is totally unimodular. If there are two rows in $Y'$ such that their support sets are incomparable, but not disjoint, the previous lemma tells us that there is a cycle in $T$ of length $\geq 4$. So the support sets of any pair of rows are either comparable or disjoint. After some reordering of the columns, this becomes an interval matrix. Lemma~\ref{lem:int} implies that $Y'$ is totally unimodular.
\end{proof}

Using this, we are going to show that when $\T_{\alpha}^*$ is a length $2$ simplex, $\Q_{\alpha}^*$ is connected.

\begin{lemma}
\label{lem:qconn}
Let $A$ and $B$ be two cells in $\Q_{\alpha}^*$ such that $\alpha = (a_1,\cdots,a_n)$ and $\sum a_i = n-2$. Then there is a path in $\Q_{\alpha}^*$ from $A$ to $B$, consisting of cells and their facets.
\end{lemma}

\begin{proof}

Any tope in $\T_{\alpha}^*$ contains at least $a_i$ number of $i$'s. And any tope that is not on the $i$-facet of $\T_{\alpha}$ contains at least $a_i+1$ number of $i$'s. Now choose any tope $C$ in $\T_{\alpha}^*$. Let $T$ be a cell that contains $C$ and intersects with the interior of $\T_{\alpha}^*$. 

We can view $T \cap \T_{\alpha}^*$ as the solution space of inequalities defining the cell $T$ and inequalities of the form $x_i \geq a_i \in \ZZ$. If we rewrite these inequalities in terms of $Yx \leq b$, then $b$ is an integer vector. And $Y$ is a totally unimodular matrix due to Corollary~\ref{cor:int}. We know that this intersection is non-empty, full-dimensional and bounded by $\T_{\alpha}$. Theorem~\ref{thm:int} tells us that the solution space is a full-dimensional integer polytope. Hence $T$ contains at least $d$ topes of $\T_{\alpha}^*$ such that for each $i$, there is at least one tope not on the $i$-facet of $\T_{\alpha}^*$. If some tope of $T$ contains $k$ number of $i$'s then $T$ also contains at least $k$ number of $i$'s. So $T$ is in $\Q_{\alpha}^*$. 

Now let $A$ and $B$ be any two cells of $\Q_{\alpha}^*$. They share at least one tope in $\T_{\alpha}^*$. We can draw a path near this tope inside $\T_{\alpha}^*$ that starts at $A$, ends at $B$ and goes through only the cells and their facets. From what we proved just before, all cells that this path goes through are cells of $\Q_{\alpha}^*$.
\end{proof}

\begin{corollary}
\label{cor:qconn}
Pick any $\alpha = (a_1,\cdots,a_n)$ and let $A$ and $B$ be two types in $\Q_{\alpha}$. Then there is a path connecting them.
\end{corollary}

\begin{proof}

Let $A,B$ be cells that are adjacent (i.e. sharing a common facet). If they are both in $\Q_\alpha$, then their common facet $(A_1\cap B_1,\dots,A_n\cap B_n)$ is also in $\Q_\alpha$. And for any type in $\Q_\alpha$, a cell that contains it is also in $\Q_\alpha$. Therefore by the surrounding property, it suffices to prove that the cells in $\Q_\alpha$ are connected by their common facets, so that the walk through adjacent facets connects every cell in $\Q_\alpha$. And by Remark~\ref{rem:dual}, it is enough to prove the existence of such walk between two cells in $\Q_{\alpha}^*$. This follows from repeatedly using Lemma~\ref{lem:qconn}.

\end{proof}

Now we are ready to prove that elimination holds.

\begin{proposition}
\label{prop:elim}
Elimination property holds for $\M$, a collection of trees coming from a fine mixed subdivision of $n\Delta_{d-1}$.
\end{proposition}
 
\begin{proof}
Let us dedicate $l_{A,B}$ to be the length of a shortest path between $A$ and $B$ in $\Q_{rk(A,B)}$. It is well defined by Corollary~\ref{cor:qconn}. We are going to show that there is a strong path between $A$ and $B$ by induction, decreasing  $rk(A,B)$ and then increasing $l_{A,B}$. 

When $rk(A,B)$ is maximal (i.e. $\sum_i rk(A,B)_i = d-1$), $A$ and $B$ have to be spanning trees. Since Proposition~\ref{prop:Gcell} tells us that $A=B$, the claim is obvious in this case. The claim is also obvious when $l_{A,B}=0$, since $\Delta(A,B) \leq l_{A,B}$.  So assume for the sake of induction, that we know there is a strong path between any pair $D,E$ such that 
\begin{itemize}
\item $rk(D,E) > rk(A,B)$ or
\item $rk(D,E) = rk(A,B)$ and $l_{D,E} < l_{A,B}$.
\end{itemize}

Let $A = C^0 \rightarrow A' = C^1 \rightarrow \cdots \rightarrow C^{l_{A,B}} = B$ be a shortest path between $A$ and $B$ in $\Q_{rk(A,B)}$. Notice that $A' \in \Q_{rk(A,B)}$ implies $rk(A',B) \geq rk(A,B)$. Then the induction hypothesis tells us that there is a strong path between $A'$ and $B$. Replace $A' \rightarrow \cdots \rightarrow B$ with the strong path between $A'$ and $B$, then we still get a shortest path between $A$ and $B$ in $\Q_{rk(A,B)}$. Now we are going to do a case-by-case analysis on how $A \rightarrow A'$ looks like.

\begin{enumerate}
\item If an element of $B_i \setminus A_i$ is added to the $i$-th coordinate, or if $A_i \supset B_i$ and an element of $A_i \setminus B_i$ is deleted from $i$-th coordinate, then this path is a strong path between $A$ and $B$.

\item Consider the case when some element $q \not \in B_i \setminus A_i$ is added to the $i$-th coordinate. We are going to show that this case cannot happen. Let $C^t \rightarrow C^{t+1}$ be the first pair of types where $q$ gets deleted from the $i$-th coordinate. Look at the path $A'=C^1 \rightarrow \cdots \rightarrow C^t$. Any type $C$ among this path should satisfy $|C_i| \geq min(|C_i^1|,|C_i^t|) > min(|A_i|,|B_i|)$. Even after we delete $q$ from the $i$-th coordinate for all types in this path, they are still in $\Q_{rk(A,B)}$. So we may replace $A' \rightarrow \cdots \rightarrow C^t$ with a path in $\Q_{rk(A,B)}$ that is strictly shorter. We get a contradiction since $A \rightarrow \cdots \rightarrow B$ is a shortest path between $A$ and $B$ in $\Q_{rk(A,B)}$.

\item The remaining case is when some element $q$ is deleted from the $i$-th coordinate where $A_i \not \supset B_i$. We are going to show that we may ignore this case. Let $C^t \rightarrow C^{t+1}$ be the first pair of types where some element $q'$ gets added to the $i$-th coordinate. Such $t$ exists since $A_i \not \supset B_i$. Notice that $C^{t+1} \in \Q_{rk(A,B)}$ implies $rk(A,C^{t+1}) \geq rk(A,B)$. Then induction hypothesis tells us that we have a strong path between $A$ and $C^{t+1}$. We can replace $A \rightarrow \cdots \rightarrow C^{t+1}$ with this strong path between $A$ and $C^{t+1}$. Then we get a path $A \rightarrow A' \rightarrow \cdots \rightarrow B$ that is a shortest path between $A$ and $B$ in $\Q_{rk(A,B)}$. As before, replace $A' \rightarrow \cdots \rightarrow B$ with a strong path between $A'$ and $B$, then we get a path that falls into one of the previous cases.
\end{enumerate}

So induction tells us that the claim is true.

\end{proof}  

We will roughly sketch how the process works. Let's assume that when going from $A$ to $A'$, the $i$-th coordinate changed. If the $i$-th coordinate of the path changes like
$$ 123 \rightarrow 1235 \rightarrow \cdots \rightarrow 14, $$

induction hypothesis on the length tells us that $A' \rightarrow \cdots \rightarrow B$ can be replaced with a strong path of same length. So now the $i$-th coordinate of the path changes like
$$ 123 \rightarrow 1235 \rightarrow 12345 \rightarrow 1245 \rightarrow 124 \rightarrow 14. $$

Using the surrounding property, we can get
$$ 123 \rightarrow 123 \rightarrow 1234 \rightarrow 124 \rightarrow 124 \rightarrow 14. $$

Then we get a redundant type in this path, so it is not a shortest-length path.

If the path changes like
$$ 123 \rightarrow 23 \rightarrow \cdots \rightarrow 14,$$

induction hypothesis tells us that $A' \rightarrow \cdots \rightarrow B$ can be replaced with a strong path of same length. So now the path changes like
$$ 123 \rightarrow 23 \rightarrow 234 \rightarrow 1234 \rightarrow 134 \rightarrow 14. $$

Induction hypothesis on the length tells us there is a strong path between $123$ and $234$, and we can replace this part to get
$$ 123 \rightarrow 1234 \rightarrow 234 \rightarrow 1234 \rightarrow 134 \rightarrow 14. $$

So for proof purposes, we could ignore the case when an element in an incomparable coordinate was deleted going from $A$ to $A'$.

\begin{corollary}
\label{cor:tom}
Given a collection of all trees in a triangulation of $\Delta_{n-1} \times \Delta_{d-1}$, it forms a tropical oriented matroid.

\end{corollary}

\section{Further Remarks}


Tropical oriented matroids are in bijection with mixed subdivisions of $n\Delta_{d-1}$. Unimodular oriented matroids are in bijection with mixed subdivisions of a zonotope, where any edge used in the summand is an edge of $\Delta_{d-1}$. There happens to be a natural class of polytopes that contains these two polytopes at the same time, which is called the \newword{generalized permutohedra} \cite{Postnikov01012009}. The trees coming from faces of a fine mixed subdivision of a generalized permutohedra are also $(n,d)$-types, so this suggests that the general framework would be similar.

The surrounding property and the comparability property still hold for generalized permutohedra. In the proof of the elimination property for $n\Delta_{d-1}$ case, all we needed was the connectivity of $\Q_{\alpha}$. And this seems to be a property that generalized permutohedra would also have, since the fact that RDV encodes the position of the cell is still true for generalized permutohedra. Boundary axiom can be modified, in the sense that the boundary topes have to be the vertices defining the convex hull of a generalized permutohedron. One major difference is that all cells of a generalized permutohedron satisfy some property called the \newword{dragon marriage condition}, which is trivial in the $n\Delta_{d-1}$ case. So we add one axiom to properly reflect this condition. Below is our definition of the \newword{generalized tropical oriented matroid}:

\begin{definition}
Let $P=P_G(y_1,\dots,y_n)=y_1\Delta_{I_1}+\dots+y_n\Delta_{I_n}$ be a generalized permutohedron, where $\Delta_{I_i}$'s are faces of $\Delta_{d-1}$ and $y_i\ge 0$ for all $i$. A collection $\mathcal{M_P}$ of $(n,d)$-types is called a \newword{generalized tropical oriented matroid} of $P$ if it satisfies the following conditions:
\begin{itemize}
 \item Boundary : For each vertex $v$ of $P$, there is unique $(n,d)$-type $(\{a_1\},\dots,\{a_n\})$ such that $v = y_1 \Delta_{\{a_1\}} + \cdots + y_n\Delta_{\{a_n\}}$.
 \item Surrounding : Same as tropical oriented matroids.
 \item Comparability : Same as tropical oriented matroids.
 \item Elimination : Same as tropical oriented matroids.
 \item Dragon Marriage : Any cell satisfies the dragon marriage condition.
\end{itemize}
\end{definition}

And our question would be:

\begin{question}
 Given a generalized permutohedron $P$, is there a bijection between the mixed subdivisions of $P$ and generalized tropical oriented matroids $\mathcal{M_P}$?
\end{question}

\bibliographystyle{plain}    
\bibliography{tom}        
 
\end{document}